\documentclass[12pt,scrartcl,a4paper]{article}

\usepackage{a4wide}
\usepackage{setspace}

\usepackage[affil-it]{authblk}

\RequirePackage[colorlinks,citecolor=blue,urlcolor=blue]{hyperref}

\usepackage{amsfonts,amsthm,amsmath,amssymb,mathtools,bbm,mathrsfs,epsfig}
\usepackage{thmtools}

\usepackage{multirow}
\usepackage[authoryear,round]{natbib}
\usepackage{booktabs}
\usepackage{float}
\usepackage{tabularx}
\usepackage{enumitem}
\usepackage{ amssymb }
\usepackage{subcaption}
\usepackage{capt-of}
\usepackage{color}
\usepackage{cancel}
\usepackage{makecell}
\numberwithin{equation}{section}

\newcommand{\R}{\mathbb{R}}

\newcommand{\N}{\mathbb{N}}

\newcommand{\E}{\mathrm {E}}

\newcommand{\PPs}{\mathcal{P}}
\newcommand{\PPf}{\mathfrak{P}}

\newcommand{\Lp}{\mathrm{L}}

\renewcommand{\d}{\mathrm{d}}
\newcommand{\1}{\mathbbm{1}}

\newcommand{\scalar}[2]{\mathchoice{\left\langle #1, #2\right\rangle}{\langle #1, #2\rangle}{\langle #1, #2\rangle}{\langle #1, #2\rangle}}

\newcommand{\I}{{ \operatorname  I}}

\declaretheorem[refname = {Theorem}]{theorem}

\declaretheorem[]{example}

\def\AA{ \mathfrak{A} }

\def\BB{ \mathfrak{B} }

\renewcommand{\epsilon}{\varepsilon}
\renewcommand{\rho}{\varrho}
\renewcommand{\theta}{\vartheta}

\newcommand\blfootnote[1]{%
  \begingroup
  \renewcommand\thefootnote{}\footnote{#1}%
  \addtocounter{footnote}{-1}%
  \endgroup
}

\makeatletter
\newcommand{\vast}{\bBigg@{4}}
\newcommand{\Vast}{\bBigg@{5}}
\makeatother

\begin{document}
\title{A unified approach for testing in Hilbert spaces on incomplete data}
\author[1,2,3]{Daniel Gaigall$^*$}
\author[1]{Philipp Wübbolding}

\affil[1]{Faculty 09 - Medical Engineering and Technomathematics, FH Aachen - University of Applied Sciences, Heinrich-Mußmann-Straße 1, 52428 Jülich, Germany}

\affil[2]{Institute for Data-driven Technologies, FH Aachen - University of Applied Sciences, Heinrich-Mußmann-Straße 1, 52428 Jülich, Germany}

\affil[3]{House of Insurance, Leibniz University Hannover, Welfengarten 1, 30167 Hannover, Germany}

\date{}

\maketitle

\blfootnote{$^*$Corresponding author. Email address: gaigall@fh-aachen.de, daniel.gaigall@insurance.uni-hannover.de}

\begin{abstract}
We consider statistical testing on the basis of incomplete observations with values in a  separable Hilbert space, where the dimension is possibly large or even infinite. The general Hilbert space setting allows various data types as they arise in modern applications, in particular high dimensional and functional data. Possible  Hilbert space testing problems are  goodness-of-fit, symmetry, homogeneity and independence. We present an approach for modeling incomplete data that covers several problems in practice, e.g., ultra high dimensional random vectors with missing entries or partially observed stochastic processes. We identify a specific structure (independent and identically distributed) in the incomplete data that enables the analysis of  statistical procedures with the help of suitable mathematical results (e.g., laws of large numbers and central limit theorems). Additionally, a general and novel concept for testing different hypotheses in this situation is suggested and sketched for the example of testing goodness-of-fit for normality.



\end{abstract}

\section{Introduction}

To allow a broad range of data types, such as multivariate data, high-frequency data, high dimensional data and functional data, we are interested in statistical inference on the basis of random variables with values in a separable Hilbert space, where the underlying dimension is possibly large or even  infinite. In this context, the treatment of various testing problems is of practical interest. Due to the fact that numerous established models suppose normality of the data, an application-oriented  problem is testing   goodness-of-fit and the hypothesis of Gaussianity, that can be considered in the multivariate or even infinite dimensional case analogously to  univariate situations. Furthermore, the testing problem of symmetry and, in the paired  or multi-sample case,  the testing problems of independence and (marginal) homogeneity are of particular interest in applications. 

\medskip
For motivation, let us consider some application problems. At first, we have that  several entries are missing in the random vectors measured, such as  in \citet{10.1093/biostatistics/kxx039}, where concentrations of potent broadly neutralizing antibodies are considered that can prevent a HIV infection at genital surfaces in males. Here, biological samples were obtained from both inner and outer foreskin from male persons and build a bivariate data set. Moreover, a monotone missingness (dropout) in random vectors can occur, such as in  \citet{Goldstein} and  \citet{DETKE2004457}, where the degree of depression before and after intake of a specific drug or a placebo is considered. Now,several patients left the study early, resulting in data sets with a low and high dropout. Furthermore, suppose partially observed functions (i.e., pieces of functions are missing)  are given as observations, similar as in  \citet{repec:bla:jorssb:v:77:y:2015:i:4:p:777-801} for ambulatory blood pressure monitoring, where participants wear a calibrated automatic device  to record systolic and diastolic blood pressure and heart rate. Some intervals have not been measured for due to  the participant’s discomfort resulting in the removal of the device or  the failure of the device to take measurements. 

\medskip
For taking  practical problems that arise with such data types into account, we consider incomplete observations in our general Hilbert space setting. We treat different models with incomplete data that cover those situations. Furthermore, the problems of ultra high dimensional random vectors and smoothed functions can  also be embed in the incomplete data approach. Ultra high dimensional random vectors arise, e.g., in the context of the investigation of the association between gut microbiome and obesity in \citet{Escobar}, where rDNA sequences are analyzed. Smoothed functions arise naturally with expansions of functions with respect to a certain basis, such as it is discussed in  \citet{10.1063/1.2981526}, where the density-matrix  of a spherically symmetric system is  expanded as a Fourier-Legendre series of Legendre polynomial with pplications to harmonically trapped electron pairs, i.e. Moshinsky's and Hooke's atoms.

\medskip
A suitable approach  for testing on the basis of complex data, such as functional observations, is to work with  projections. In the context of testing goodness-of-fit, \citet{Random_projections} and \citet{CUESTAALBERTOS20074814} apply random projections to obtain real-valued observations and finally Kolmogorov-Smirnov type tests. \citet{RePEc:eee:jmvana:v:100:y:2009:i:4:p:753-766} extend this procedure to  more general spaces and  \citet{RePEc:spr:testjl:v:19:y:2010:i:3:p:537-557} apply the idea  to an ANOVA test for functional data.  Moreover, a bootstrap based goodness-of-fit test for functional data is proposed by  \citet{9e51bff2-c3de-37cf-b969-bf2517f82c90}, where the null hypothesis 
consists of a parametric family of distributions, and \citet{https://doi.org/10.1002/jae.2846} treat an extension to the multi-sample case. \citet{DitzhausGaigallConsistentTest}, \citet{DisthausGaigallMarginalHomo} and \citet{GaigallLiangWu-HilbertIndependence} develop the projection idea for the testing problems of goodness-of-fit, marginal homogeneity and independence in separable Hilbert spaces and  propose tests  by integrating over all real-valued projections. A review on goodness-of-fit tests and testing in regression models with functional data is given in \citet{GonzlezManteiga2022ARO}.

\medskip
The concept of the characteristic function is also available in Hilbert spaces. Moreover, the characteristic functional uniquely determines the distribution, see \citet{LahaRohatgi}. For that reason, test statistics based on the empirical characteristic function, and related $L^2$-distances,  promise consistent testing procedures.  This testing idea goes back to the common Baringhaus-Henze-Epps-Pulley (BHEP) test, see \citet{BaringhausHenze}, that extends the idea of \citet{Epps-Pulley} to the multivariate case and serves as a powerful goodness-of-fit test for multivariate normality. The common BHEP test requires complete random vectors as observations and is not applicable in higher dimension, i.e., if the dimension of the random vectors is much larger than the sample size. A review of tests for multivariate normality with emphasis on  $L^2$ type statistics is given in \citet{RePEc:spr:testjl:v:29:y:2020:i:4:d:10.1007_s11749-020-00740-0}. In \citet{HenzeGamero}, the BHEP test is generalized  to the separable Hilbert space case. Based on the latter work,  \citet{CDHH} provide an extension to a test for subfamilies of Gaussian distributions and \citet{GaigallWuebbolding_gBM-Test} propose a BHEP type goodness-of-fit test for the geometric Brownian motion in the functional data setting. 

\medskip
The idea of the BHEP test can  be adopted for the treatment of other testing problems; examples in the case of real-valued observations are given in \citet{RePEc:spr:stpapr:v:57:y:2016:i:4:d:10.1007_s00362-016-0760-0} for a more general goodness-of-fit problem for semiparametric and parametric hypotheses, in \citet{BakirovRizzoSzekely-Inditest-06} for the testing problem of independence or in \citet{HENZE2003275} for the testing problem of symmetry.  Based on characteristic functions, testing procedures  for  Hilbert space valued observations are investigated in \citet{dette2024newenergydistancesstatistical}. Except for the univariate case,  resampling procedures are often suggested to obtain critical values of a nonparametric test. \citet{Romano-Bootstrap-Rand-nonpar-89} presents bootstrap procedures for different testing problems in nonparametric statistics. More examples of parametric and nonparametric bootstrap procedures for the background of the testing problems of goodness-of-fit, symmetry, exchangeability, and independence for random vectors are given in \citet{Gaigall03102021}, \citet{BARINGHAUS2023105154} and  \citet{Gaigall-ApplicabilityServeralTests}. Considering nonparametric hypothesis testing on the basis of incomplete data,  \citet{CuparicMilosevic-adapt-impute} discuss a model-free imputation procedure for testing goodness-of-fit with randomly right-censored data, where the focus of \citet{Aleksić09122024} is especially on testing  multivariate normality on the basis of missing values.  \citet{Gaigall-MargHom-Incompl-data} treats the fully nonparametric testing problem of marginal homogeneity on the basis of possibly incomplete paired data with values in $\R\times \R$ and introduces a new bootstrap procedure for handling this special situation.   A BHEP test for multivariate normality on incomplete data vectors is developed and investigated in \citet{g+w-neu}. First ideas how to deal with incomplete data in  general Hilbert spaces are given in \citet{DitzhausGaigallConsistentTest} and \citet{GaigallLiangWu-HilbertIndependence}. Nonparametric density estimation on the basis of noisy functional data is the topic in \citet{0fdf3827-ba20-3d40-8bdd-32b0388081ee}.

\medskip
In Section \ref{testing}, testing in Hilbert spaces on incomplete data is discussed. We consider  a unified Hilbert space approach that covers various  applications and we identify and use a specific structure (independent and identically distributed) in the incomplete data that enables the analysis of  statistical procedures with the help of suitable mathematical results (e.g., laws of large numbers and central limit theorems). Additionally, a general and novel concept for testing different hypothesis in this situation is suggested and explained. In detail,  the  Hilbert space setting is introduced in Subsection \ref{gen}. Subsection \ref{mod} treats modelling incomplete data in the Hilbert space. Testing problems in Hilbert spaces are presented in Subsection \ref{tp}. Subsection \ref{proj}  investigates a suitable  projection approach. The testing idea is outlined in Subsection \ref{idea}. For illustration, Subsection \ref{exanorm}  gives an example for testing goodness-of-fit for  normality on incomplete data  from  \citet{g+w-neu}. Section \ref{outlook} closes with motivations and hints  for further  research.

\section{Testing in Hilbert spaces on incomplete data}\label{testing}

\subsection{General Hilbert space setting}\label{gen}

We consider a general Hilbert space approach, such as in \citet{DitzhausGaigallConsistentTest}, \citet{DisthausGaigallMarginalHomo} and \citet{GaigallLiangWu-HilbertIndependence}. 
Let $H$ be a separable Hilbert space of arbitrary dimension. The Hilbert space is equipped with a scalar product $\scalar{x}{x'}$, $(x,x')\in H^2$, that induces a norm $\lVert x\rVert=\scalar{x}{x}^\frac{1}{2}$, $x\in H$. Let $(e_j)_{j\in J}$ be a countable orthonormal basis of the separable Hilbert space $H$, where $e_j$ is the $j$-th basis element,  $j\in J$, and the index set is given by $J=\{1,\dots,d\}$ for some $d\in\N$ in the finite dimensional case  or $J=\N$ in the infinite dimensional case. For $x\in H$ and $j\in J$, we denote by $x(j)=\scalar{x}{e_j}$ the $j$-th coordinate of $x$. On $H$, we have open sets in the norm topology that generate the Borel $\sigma$-field $\BB(H)$, which subsequently allows for the notion of random variables taking values in $H$. Let $(\Omega,\AA,P)$ be a probability space and $X:\Omega\rightarrow H$ be a random variable with values in $H$. Denote by $P^X$  the distribution of $X$.  In what follows, let $X_i$, $i=1,\dots,n$, be a sample of size $n\in\N$ from $P^X$, that are independent copies of $X$. The Hilbert space approach allows a broad range of data types. By resorting to the Hilbert space setting, we are able to deal with  various data types. Examples are given as follows.

\begin{example}
\label{ex:Hilbertraeume}
a) Random vectors or random sequences: Suppose $d\in\N$ and that we deal with random vectors $X_i=(X_i(1),\dots,X_i(d))$, $i=1,\dots,n$, with  values in $\R^d$. Define  $J=\{1,\dots,d\}$ and  $\R^J=\{f; f:J\rightarrow\R\}$. Moreover,  denoting by $\delta_j$  the Dirac delta or the Dirac measure in $j\in J$, we introduce with $\eta_J=\sum_{j\in J}p_j\delta_j$  a measure on $\PPf(J)$, the power set of $J$, where   $p_j>0$ for all $j\in J$  and $\sum_{j\in J}p_j<\infty$. Then the random vectors can be seen as random variables with values in  the separable Hilbert space 
    \begin{equation*}
        H= \ell^2(J,\PPf(J),\eta_J),
    \end{equation*}
that is the space of all $x\in\R^J$ that satisfy $\sum_{j\in J}x(j)^2p_j<\infty$, equipped with the scalar product $\scalar{x}{x'}=\sum_{j\in J}x(j)x'(j)p_j$, $(x,x')\in H^2$,  and $e_j=\delta_j/\sqrt{p_j}$, $j\in J$, defines a countable orthonormal basis.  In fact, this point of view enables an extension to the infinite dimensional case with $J=\N$ and $X_i=(X_i(j))_{j\in \N}$, $i=1,\dots,n$, as random sequences. Note that, in the finite dimensional case, the choice $p_1=\dots=p_d=1$ leads to the standard scalar product in $\R^d$.

\medskip    
b) Functional data: Let $A\subset \R$ be a  non-degenerated and bounded interval and suppose we deal with stochastic processes $X_i$, $i=1,\dots,n$, with values in the separable Hilbert space 
$$H=\Lp^2 (A,\BB(A),\lambda_{A}),$$
that is the space of equivalence classes of measurable functions $x:A\rightarrow\R$ that satisfy $\int_A x^2\d \lambda_{A}<\infty$. The scalar product is given by $\scalar{x}{x'}=\int_A x x'\d \lambda_{A}$, $(x,x')\in H^2$. A countable orthonormal basis $(e_j)_{j\in\N}$ of $H$ is available by  the trigonometric orthonormal basis, for instance. 
   
\end{example}

\subsection{Modelling incomplete data in the Hilbert space setting}\label{mod}

Expanding our random variable  with respect to the basis  we obtain
\begin{equation*}
    X=\sum_{j\in J}\scalar{X}{e_j}e_j=\sum_{j\in J}X(j)e_j.
\end{equation*}
The latter representation  motivates  to model incomplete data with the help of another random variable $I:\Omega\rightarrow H$ that models the missing mechanism. This random variable has the property that $I(j)=\scalar{I}{e_j}$ takes only values in $\{0,1\}$ for all $j\in J$, i.e.,  $I(j)$ is a binary (Bernoulli) random variable. We introduce the  Hadamard product of $I$ and $X$ by 
\begin{equation*}
    I\odot X = \sum_{j\in J} \scalar{I}{e_j}\scalar{X}{e_j}e_j=  \sum_{j\in J} I(j)X(j)e_j.
\end{equation*}
In the finite dimensional case of random vectors, the notation ``$\odot$'' coincides with the common Hadamard vector (or matrix) product. Our incomplete data are given by 
$$(I_i,I_i\odot X_i),~i=1,\dots,n,$$
that are independent copies of $(I,I\odot X)$ with unknown distribution $P^{(I,I\odot X)}$. This framework enables to address a broad range of incomplete data types. Clearly, the complete data case is included by setting $I(j)=1$, $j\in J$. Moreover, monotone missingness (dropout) effects  are  covered. Moreover, the case of missing entries in  random vectors   is  covered, such as in \citet{Gaigall-MargHom-Incompl-data} for the testing problem of marginal homogeneity on the basis of paired data with values in $\R\times \R$ or in  \citet{g+w-neu}, where this data framework is  used for testing goodness-of-fit for multivariate normality on the basis of random vectors in $\R^d$. Another important setting included is given by partially observed functions, i.e., pieces of the functions observed
are missing. Furthermore, the problem that  smoothed functions (instead of the true un-smoothed functions) are given as observations can  also be embed in this  incomplete data approach. The same applies to the case of ultra high dimensional random vectors;  here, we have $d\gg n$ and we consider a dimension $d=d_n$ depending on the sample size such that  asymptotics of our testing procedures can be studied in the case that  the dimension $d_n\to\infty$ tends to infinity as the sample size $n\to\infty$ tends to infinity; compare with \citet{GaigallLiangWu-HilbertIndependence}. The following examples give more details.

\begin{example}
\label{ex:Incompleteness}
a) Missing entries in   random vectors: Consider the case of random vectors $X_i=(X_i(1),\dots,X_i(d))$, $i=1,\dots,n$, with values in $H= \ell^2(\{1,\dots,d\},\PPf(\{1,\dots,d\}),\eta_{\{1,\dots,d\}})$ and   orthonormal basis $(e_j)_{j\in \{1,\dots,d\}}$ in Example \ref{ex:Hilbertraeume} a).   Then our incomplete observations  $(I_i,I_i\odot X_i)$, $i=1,\dots,n$,   satisfy
    \begin{equation*}
       I_i\odot X_i=\sum_{j\in J}I_i(j)X_i(j)e_j= \sum_{\substack{j\in\{1,\dots,d\}\\I_{i}(j)=1}}X_i(j)e_j,
    \end{equation*}
that are  random vectors $(I_i(1)X_i(1),\dots,I_i(d)X_i(d))$   with $I_i(j)X_i(j)=X_i(j)$ if $I_{i}(j)=1$ and $I_i(j)X_i(j)=0$ if $I_{i}(j)=0$ and with values again in $H$. This is a possible specification  in our incomplete data approach that models   missing entries in   random vectors.

\medskip
b) Monotone missingness (dropout) in random vectors:  Consider  the case of random vectors $X_i=(X_i(1),\dots,X_i(d))$, $i=1,\dots,n$, with values in $H= \ell^2(\{1,\dots,d\},\PPf(\{1,\dots,d\}),\eta_{\{1,\dots,d\}})$ and   orthonormal basis $(e_j)_{j\in \{1,\dots,d\}}$ in Example \ref{ex:Hilbertraeume} a). Let $D_i$, $i=1,\dots,n$, be independent and identically distributed random variables with values in $\{1,\dots,d\}$. Set $I_{i}(1)=\dots=I_{i}(D_i)=1$ and  $I_{i}(D_i+1)=\dots=I_{i}(d)=0$, $i=1,\dots,n$. Then our incomplete observations  $(I_i,I_i\odot X_i)$, $i=1,\dots,n$,   satisfy
    \begin{equation*}
       I_i\odot X_i=\sum_{j\in J}I_i(j)X_i(j)e_j= \sum_{j=1}^{D_i}X_i(j)e_j,
    \end{equation*}
that are  random vectors $(X_i(1),\dots,X_i(D_i),0,\dots,0)$  with values again in $H$. This is a possible specification  in our incomplete data approach that models  monotone missingness (dropout) in   random vectors.

\medskip
c) Partially observed functions: Let $A\subset \R$ be a non-degenerated and  bounded interval and consider stochastic processes $X_i$, $i=1,\dots,n$, with values in  $H=\Lp^2 (A,\BB(A),\lambda_{A})$   introduced in Example \ref{ex:Hilbertraeume} b). Suppose we have  a partition of $A$ in  $k\in\N$ intervals. For the sake of clarity, we consider only the case of $k=2$ intervals $A_1\subset A$ and $A_2\subset A$ with $A_1\cap A_2=\emptyset$  and $A_1\cup A_2=A$. Let  $(e_{j,1})_{j\in \N}$ and $(e_{j,2})_{j\in \N}$ be orthonormal bases  of  $H_1=\Lp^2 (A_1,\BB(A_1),\lambda_{A_1})$ and $H_2=\Lp^2 (A_2,\BB(A_2),\lambda_{A_2})$. Denoting by ``$\1$'' the indicator and  writing  $e_{j,1}=e_{j,1}\1_{A_1}$ and $e_{j,2}=e_{j,2}\1_{A_2}$, these functions can be extended to functions in $H$ in the obvious way. Defining the functions  $e_{2j-1}=e_{j,1}\1_{A_1}$ and $e_{2j}=e_{j,2}\1_{A_2}$, $j\in \N$, it is     $(e_j)_{j\in \N}$ an orthonormal basis of $H$.  Let  $I_{i}(1)=I_{i}(3)=\dots$ and $I_{i}(2)=I_{i}(4)=\dots$, $i=1,\dots,n$.  Then our incomplete observations  $(I_i,I_i\odot X_i)$, $i=1,\dots,n$,   satisfy
    \begin{equation*}
       I_i\odot X_i=\sum_{j\in J}I_{i}(j)X_i(j)e_j= I_{i}(1)\1_{A_1}\sum_{j=1}^{\infty}X_i(2j-1)e_{j,1}+I_{i}(2)\1_{A_2}\sum_{j=1}^{\infty}X_i(2j)e_{j,2},
    \end{equation*}
that are stochastic processes $I_i\odot X_i$ with values again in $H$, where pieces of the functions are missing. This approach covers the case of partially observed functions.

\medskip
d) Ultra high dimensional random vectors: Consider random sequences
 $X_i=(X_i(j))_{j\in \N}$, $i=1,\dots,n$,  with values in    $H= \ell^2(\N,\PPf(\N),\eta_{\N})$ and  orthonormal basis $(e_j)_{j\in \N}$ introduced in Example \ref{ex:Hilbertraeume} a). Let $d=d_n\in\N$ be a dimension depending on $n$ with $d_n\to \infty$ as  $n\rightarrow\infty$. Define  $I_{i}(j)=I_{i,n}(j)=1$ if $j\le d_n$ and $I_{i}(j)=I_{i,n}(j)=0$ if $j> d_n$, $i=1,\dots,n$, $j=1,\dots,d_n$. Then our incomplete observations  $(I_i,I_i\odot X_i)=(I_{i,n},I_{i,n}\odot X_i)$, $i=1,\dots,n$, satisfy
    \begin{equation*}
      I_{i,n}\odot X_i=\sum_{j\in J}I_{i,n}(j)X_i(j)e_j= \sum_{j=1}^{d_n}X_i(j)e_j,
    \end{equation*}
that  are random vectors $(X_i(1),\dots,X_i(d_n))$  with values in the finite dimensional separable Hilbert space  of dimension $d_n$ expanded by the orthonormal basis $(e_j)_{j\in \{1,\dots,d_n\}}$. This approach can cover the  case of ultra-high dimensional random vectors with $d_n\gg n$.

\medskip
e) Smoothed functions:  Consider stochastic processes $X_i$, $i=1,\dots,n$, with values in  $H=\Lp^2 (A,\BB(A),\lambda_{A})$  and  orthonormal basis $(e_j)_{j\in \N}$ introduced in Example \ref{ex:Hilbertraeume} b). Let $d\in\N$ be a dimension and define  $I_{i}(j)=1$ if $j\le d$ and $I_{i}(j)=0$ if $j> d$. Then our incomplete observations  $(I_i,I_i\odot X_i)$, $i=1,\dots,n$,   satisfy
    \begin{equation*}
       I_i\odot X_i=\sum_{j\in J}I_{i}(j)X_i(j)e_j= \sum_{j=1}^{d}X_i(j)e_j,
    \end{equation*}
that are stochastic processes  $I_i\odot X_i$ with values in the finite dimensional separable Hilbert space  of dimension $d$ expanded by the orthonormal basis $(e_j)_{j\in \{1,\dots,d\}}$. This approach covers the case of smoothed functions.

\end{example}

\subsection{Testing problems in Hilbert spaces}\label{tp}

There are different testing problems that are of particular interest in applications. The testing problems are  formulated for the underlying distributions of the random variables with values in the Hilbert space. On the basis of the incomplete data $(I_i,I_i\odot X_i)$, $i=1,\dots,n$, introduced above, we consider the testing problem of goodness-of-fit whether the underlying distribution $P^X$ belongs to a given family of distributions $\PPf^X=\{P^X_\vartheta;\vartheta\in\Theta\}$, where $\PPf=\{P_\vartheta;\vartheta\in\Theta\}$ is a family of probability measures on $\AA$ and $\Theta$ is a non-empty parameter set, that is the testing problem
\begin{equation*}
    H_0: P^X\in\PPf^X~~\text{vs.}~~ H_1:P^X\notin\PPf^X.
\end{equation*}
For example, $\PPf^X$ could be the family of multivariate normal distributions, such as in  \citet{g+w-neu}, where the case of incomplete data vectors is treated. Another example of $\PPf^X$ is a (sub-)family of Gaussian distributions, such as in \citet{GaigallWuebbolding_gBM-Test} for the geometric Brownian motion, where the complete functional data setting is treated. Another interesting one-sample testing problem that can be treated  on the basis of the incomplete data $(I_i,I_i\odot X_i)$, $i=1,\dots,n$, is the testing problem of symmetry, that is 
\begin{equation*}
    H_0: P^X=P^{-X}~~\text{vs.}~~H_1: P^X\not = P^{-X}.
\end{equation*}
In the case of complete data with values in $\R^d$, $d\in\N$, the testing problem of symmetry is treated in  \citet{Gaigall-ApplicabilityServeralTests}. A testing problem of interest in the two-sample case is the testing problem of  homogeneity. Here, we consider two samples of incomplete data $(I_i,I_i\odot X_i)$, $i=1,\dots,n$, and $(I_i',I_i'\odot X_i')$, $i=1,\dots,n'$, with sample sizes $n$ and $n'$. The  related underlying distributions  of the random variables  $X_i$, $i=1,\dots,n$, and $X_i'$, $i=1,\dots,n'$, are denoted by $P^X$ and $P^{X'}$. Supposing that the two samples of incomplete data are independent, the testing problem of homogeneity is given by 
\begin{equation*}
    H_0:P^X= P^{X'}~~\text{vs.}~~H_1:P^X\not = P^{X'}.
\end{equation*}
In the case $n=n'$ and that we have instead of two samples of incomplete data a paired sample of incomplete data $((I_i,I_i\odot X_i),(I_i',I_i'\odot X_i'))$, $i=1,\dots,n$, we deal with the testing problem of marginal homogeneity, such as in \citet{DisthausGaigallMarginalHomo} for complete  paired data with values in a  Hilbert space or in \citet{Gaigall-MargHom-Incompl-data} for incomplete paired data with values in $\R\times \R$. A related testing problem, treated again on the basis of a paired sample of incomplete data $((I_i,I_i\odot X_i),(I_i',I_i'\odot X_i'))$, $i=1,\dots,n$, is the testing problem of independence, that is the testing problem
\begin{equation*}
    H_0:P^{(X,X')}=P^X\otimes P^{X'}~~\text{vs.}~~H_1:P^{(X,X')}\not =P^X\otimes P^{X'}.
\end{equation*}
Here, $P^{(X,X')}$ denotes the joint distribution of $(X,X')$ and ``$\otimes$'' stands for the product measure. The testing problem of independence is also treated in  \citet{Gaigall-ApplicabilityServeralTests} in the case of complete data with values in $\R^d\times \R^d$, $d\in\N$, and in \citet{GaigallLiangWu-HilbertIndependence} for random variables in a Hilbert space. The mentioned testing problems can be generalised to the multi-sample case. Of course, the  treatment of additional testing problems is possible.

\subsection{Projection approach}\label{proj}

For handling with high or even infinite dimensions, a projection idea for Hilbert spaces is developed in
\citet{DitzhausGaigallConsistentTest}, \citet{DisthausGaigallMarginalHomo}, and \citet{GaigallLiangWu-HilbertIndependence}. The approach uses projections from the Hilbert space to the real line. We modify this projection idea for the incomplete data case. Later, we will see that the idea has potential for the construction of powerful tests based on incomplete data in our Hilbert space setting. We introduce 
 $$\PPs(H)=\bigg\{\pi;\pi:H\rightarrow\pi(H),\pi(\cdot)=\sum_{j\in J'}\scalar{\cdot}{e_j}e_j, J'\subset J, |J'|<\infty\bigg\},$$
 that is  a set of projections from $H$ onto the finite dimensional subspaces  of $H$. 

\begin{theorem}
    \label{prop:Equal-missingData}
    Consider a separable Hilbert space $H$ of arbitrary dimension. For all $\pi\in\PPs(H)$  let  the probability that all components of $\pi(I)$ are equal to one be different from zero and let $X$ and $I$ as well as  $X'$ and $I$ be independent. Then we have the equivalence
 $$P^X= P^{X'}~\Leftrightarrow~~\forall \pi\in\PPs(H): P^{\pi(I\odot X)}= P^{\pi(I\odot X')}.$$ 
\end{theorem}

\begin{proof}
With similar arguments as in  \citet{DitzhausGaigallConsistentTest}, we obtain that  the projections in $\PPs(H)$ determines the  distribution $P^X$ of $X$ uniquely. Precisely,  for all random variables  $X$ and $X'$ in $H$ we have the equivalence
 $$P^X= P^{X'}~~\Leftrightarrow ~~\forall \pi\in\PPs(H): P^{\pi (X)}= P^{\pi (X')}.$$ 
Furthermore, the  distribution of the incomplete data $P^{I\odot X}$  determines the  distribution $P^X$ of $X$ uniquely. From our previous deliberation, it is sufficient to consider only the case  that $J=\{1, \dots,d\}$ for some $d\in\N$, where $P(I(1)=1,\dots,I(d)=1)>0$ is satisfied  and  that $X$ and $I$ as well as  $X'$ and $I$  are independent. Analogously to \citet{g+w-neu}, it is
 $$P(X(1)\le x_1,\dots,X(d)\le x_d)=\frac{P(I(1)X(1)\le x_1,\dots,I(d)X(d)\le x_d,I(1)=1,\dots,I(d)=1)}{P(I(1)=1,\dots,I(d)=1)}$$ 
for all $x_1,\dots,x_d\in\R$, and so the following equivalence
 $$P^X= P^{X'}~~\Leftrightarrow~~P^{I\odot X}= P^{I\odot X'}$$ 
is valid. Combining the  findings we obtain the  result.
 \end{proof}

\subsection{Testing idea}\label{idea}

Under the stated assumptions, we deduce from Theorem \ref{prop:Equal-missingData} that our testing problems under consideration,  written in the form null hypothesis versus alternative, 
$$H_0~~\text{vs.}~~H_1,$$
can be expressed equivalently in the form  
$$\forall \pi\in\PPs(H): H_0(\pi)~~\text{vs.}~~\exists \pi\in\PPs(H): H_1(\pi),$$ 
with suitable testing problems $H_0(\pi)$ vs. $H_1(\pi)$ depending on the projections  $\pi\in\PPs(H)$. Because $\PPs(H)$ consists of projections from $H$ onto the finite dimensional subspaces  of $H$, the latter  testing problems are formulated  in finite dimensional cases and  can be treated  on the basis of the the projected incomplete  data. For illustration, let us consider the following example.

\begin{example}\label{exagof}
Let us go back to the testing problem of goodness-of-fit
whether the underlying distribution $P^X$ belongs to the family of distributions $\PPf^X=\{P^X_\vartheta;\vartheta\in\Theta\}$, that is the testing problem
\begin{equation*}
    H_0: P^X\in\PPf^X~~\text{vs.}~~ H_1:P^X\notin\PPf^X.
\end{equation*}
Introducing the families of distributions $\PPf^{\pi(I\odot X)}=\{P^{\pi(I\odot X)}_\vartheta;\vartheta\in\Theta\}$, $\pi\in\PPs(H)$, and defining the testing problems 
\begin{equation*}
    H_0(\pi): P^{\pi(I\odot X)}\in\PPf^{\pi(I\odot X)}~~\text{vs.}~~ H_1(\pi): P^{\pi(I\odot X)}\notin\PPf^{\pi(I\odot X)},
\end{equation*}
Theorem \ref{prop:Equal-missingData} shows that the testing problem  $H_0$ vs. $H_1$ can be expressed equivalently in the form  $\forall \pi\in\PPs(H): H_0(\pi)$ vs. $\exists \pi\in\PPs(H): H_1(\pi)$, where the  testing problems  $H_0(\pi)$ vs. $H_1(\pi)$,  $\pi\in\PPs(H)$, can be treated on the basis of the  projected incomplete  data
$$(\pi(I_i),\pi(I_i\odot X_i)),~i=1,\dots,n,~\pi\in\PPs(H),$$ 
respectively.
\end{example}

Motivated by the previous deliberations, we consider test statistics of the form 
$$T_n=\int\limits_{\PPs(H)} T_n(\pi)\d Q(\pi)=\sum_{\pi\in \PPs(H)}q(\pi)T_n(\pi).$$
For $\pi\in\PPs(H)$ it is $T_n(\pi)$  a test statistic for the testing problem  $H_0(\pi)$ vs. $H_1(\pi)$ based on the  projected incomplete  data $(\pi(I_i),\pi(I_i\odot X_i))$, $i=1,\dots,n$, where we assume that  $T_n(\pi)$ is non-negative and large values indicate the validity of the alternative $H_1(\pi)$. Moreover, $Q$ is a (discrete) probability measure on $\PPs(H)$ or $q:\PPs(H)\rightarrow [0,1]$ is a weight function that satisfies $\sum_{\pi\in\PPs(H)}q(\pi)=1$, that is  the probability mass function of $Q$ with  $q(\pi)=Q(\{\pi\})$, $\pi\in\PPs(H)$. $Q$ or $q$ can be chosen by the statistician. Similar as in \citet{GaigallWuebbolding_gBM-Test} in the functional data case, testing approaches based on the characteristic function are of particular interest. Given a random variable $Y$ with values in the Hilbert space $H$, the characteristic function of $Y$ is given by
$\varphi_Y(t)= E(\exp(i\scalar{t}{Y})),~t\in H$.
Note that the characteristic function determines the distribution $P^Y$ uniquely. Given independent copies $Y_\ell$, $\ell=1,\dots,n$, of $Y$, the characteristic function can be estimated by the empirical characteristic function of the $Y_\ell$, $\ell=1,\dots,n$, that is 
\begin{equation*}
    \varphi_n(t)=\frac{1}{n}\sum_{\ell=1}^n\exp(i\scalar{t}{Y_\ell}),~t\in H.
\end{equation*}
In applications, the test statistic can be evaluated with the help of a Monte-Carlo simulation. For defining $Q$  and to obtain realizations  $\pi\in \PPs(H) $ for the implementation of the Monte-Carlo procedure, it is possible to use a similar approach as in \citet{DitzhausGaigallConsistentTest},  \citet{DisthausGaigallMarginalHomo}, \citet{GaigallLiangWu-HilbertIndependence} and \citet{GaigallWuebbolding_gBM-Test}. For this purpose, we choose (discrete) distributions  $\nu$ and $\eta$ on $J$. Then, a realization  $\pi$ of  $Q$  is obtained as follows:
\begin{enumerate}
    \item[1.] Generate a realization $k$ of the distribution  $\nu$.
    \item[2.] Independently of the previous step, generate $i_1,\dots,i_k$ by sampling without replacement according to the distribution $\eta$.
    \item[3.] Set $\pi(\cdot) = \sum_{j=1}^k \scalar{\cdot}{e_{i_j}}e_{i_j}$. 
\end{enumerate}

Depending on the significance level $\alpha\in(0,1)$, a $(1-\alpha)$-quantile $c_{n,1-\alpha}$ of the  distribution of the test statistic $T_n$ under the null hypothesis is requested as critical value for the implementation of the test. Unfortunately, the test statistic is in general not distribution free under the null hypothesis, i.e., even under $H_0$, the distribution of  $T_n$ is unknown or depends on unknown parameters. This applies in particular in multivariate, high dimensional or infinite dimensional cases and in situations where the data are incomplete. Consequently, $c_{n,1-\alpha}$ is not available as a critical value in applications. To resolve this problem,  resampling techniques are probably suitable. Let us consider  parametric bootstrap procedures, such as  in \citet{Gaigall03102021}, \citet{BARINGHAUS2023105154} and  \citet{Gaigall-ApplicabilityServeralTests}, for instance. Here, a bootstrap version of the test statistic $T_n^*$ is used that is based on a bootstrap sample. The bootstrap approach estimates underlying distributions on the basis of the observations under the premise that the null hypothesis is true. Denote by  $C_{n,1-\alpha}^*$ a $(1-\alpha)$-quantile of the conditional distribution of $T_n^*$ given the incomplete data $(I_i,I_i\odot X_i)$, $i=1,\dots,n$. This is a random variable based on the data and is available as critical value in applications. In practice, $C_{n,1-\alpha}^*$ is obtained by Monte-Carlo simulation.  The test is given by $\varphi=\1(T_n>C_{n,1-\alpha}^*)$,
i.e., the test rejects the null hypothesis if and only if $T_n>C_{n,1-\alpha}^*$.

\subsection{Example: Testing goodness-of-fit for  normality on incomplete data}\label{exanorm}

For illustration, we scetch the testing problem of goodness-of-fit for  normality on the basis of incomplete data in  \citet{g+w-neu}. There,   a modification of the common BHEP test applicable for random vectors with  missing entries is treated, compare with Example \ref{ex:Hilbertraeume} a) and Example \ref{ex:Incompleteness}  a). We note that the testing approach covers  the case that the dimension of the random vectors is  larger than the sample size. Let $H=\R^d$ equipped with the standard  scalar product $\scalar{x}{x'}=\sum_{j=1}^dx(j)x'(j)$, $(x,x')\in H^2$, and the standard orthonormal basis $e_1=(1,0,\dots,0),\dots,e_d=(0,\dots,0,1)$. Here, we deal with random vectors $X=(X(1),\dots,X(d))$ and $I=(I(1),\dots,I(d))$. We consider the testing problem of goodness-of-fit $H_0: P^X\in\PPf^X$ vs. $H_1:P^X\notin\PPf^X$, where  
$$\PPf^X=\{N_d(\mu,\Sigma);\mu\in\R^d,\Sigma\in\R^{d\times d}~\text{symmetric positive definite}\}$$
is the family of $d$-dimensional normal distributions and $\PPf^{\pi(I\odot X)}$, $\pi\in\PPs(\R^d)$, are the corresponding families of $k$-dimensional normal distributions, $k=1,\dots,d$, compare with Example \ref{exagof}. In what follows, we suppose the assumptions  from Theorem \ref{prop:Equal-missingData} and that $E(\lVert X\rVert^2)<\infty$.  For $a\in\{0,1\}^d$, we define by  $p(a)=P(I=a)$ the probability mass function of $I$ and write $\mu$ and $\Sigma$ for the expectation vector and the covariance matrix of $X$.  Adapting the application of  $\pi\in\PPs(\R^d)$ to  probability mass functions and matrices, it is $\pi(p)$  the probability mass function of $\pi(I)$ and   $\pi(\Sigma)$  the covariance matrix of $\pi(X)$. For $k=1,\dots,d$ and $a\in\{0,1\}^k$,  let  $D_a=\operatorname{diag}(a)\in\R^{k\times k}$ be the diagonal matrix with  diagonal entries $a$. Now, we fix some  $\pi\in\PPs(\R^d)$ with $\pi(\cdot)=\sum_{j\in J'}\scalar{\cdot}{e_j}e_j$ and  $k=|J'|$. The characteristic function of  $\pi(I\odot (X-{\mu}))$ is
\begin{align*}
\varphi_{\pi}(t)=\E\Big(\exp\big(i\scalar{t}{\pi(I\odot (X-{\mu}))}\big)\Big),~t\in \R^{k},
\end{align*}
in general and
\begin{align*}
    \phi_{\pi}(t)
    =\sum_{a\in\{0,1\}^{k}}\exp\left(-\frac{1}{2}\scalar{D_a \pi(\Sigma) D_a t}{t} \right) \pi(p)(a),~t\in\R^{k},
\end{align*}
under the null hypothesis $H_0$, see \citet{g+w-neu}. As a  consequence of  Theorem \ref{prop:Equal-missingData} and  the fact that the characteristic function determines the distribution uniquely, we have that  $ \varphi_{\pi}=\phi_{\pi}$ for all $\pi\in\PPs(\R^d)$ if and only if the null hypothesis $H_0$ is valid. Let us introduce estimators of $\mu$, $\Sigma$ and $p(a)$,  $a\in\{0,1\}^d$, based on the incomplete data $(I_i,I_i\odot X_i)$, $i=1,\dots,n$. Let $\widehat \mu_n$ be a random vector with values in $\R^d$, let $\widehat \Sigma_n$ be  a random symmetric and positive semidefinite  matrix  with values in $\R^{d\times d}$   and let $\widehat p_n(a)$,  $a\in\{0,1\}^d$, be a random probability mass function. Indeed, several choices for the estimators are possible, see \citet{g+w-neu}. Following \citet{g+w-neu}, we define the estimated empirical characteristic function
\begin{equation*}
  \widehat  \varphi_{n,\pi}(t)=\frac{1}{n}\sum_{j=1}^n\exp\big(i\scalar{t}{\pi(I_k\odot (X_k-{\widehat\mu_n}))}\big),~t\in \R^{k},
\end{equation*}
and the estimated characteristic function  under the null hypothesis $H_0$
\begin{align*}
   \widehat \phi_{n,\pi}(t)
    =\sum_{a\in\{0,1\}^{k}}\exp\left(-\frac{1}{2}\scalar{D_a \pi(\widehat\Sigma_n) D_a t}{t} \right) \pi(\widehat p_n)(a),~t\in\R^{k}.
\end{align*}
Denoting by $\Phi_k$  the $k$-dimensional standard normal distribution and setting 
\begin{equation*}
    T_n(\pi)=n\int|  \widehat  \varphi_{n,\pi}(t) -  \widehat \phi_{n,\pi}(t)|^2\d\Phi_k(t),
\end{equation*}
we obtain the BHEP  test statistic in  \citet{g+w-neu} as 
$$T_n=\sum_{\pi\in \PPs(\R^d)}q(\pi)T_n(\pi).$$
For the demand of critical values and  the implementation of the test, a bootstrap procedure is available, see  \citet{g+w-neu} for details.

\section{Outlook}\label{outlook}

Smilar as in \citet{Gaigall03102021}, our test statistics $T_n=\sum_{\pi\in \PPs(H)}q(\pi)T_n(\pi)$ are weighted sums (in the finite dimensional case) or even   series (in the infinite dimensional case) of different test statistics $T_n(\pi)$, $\pi\in\PPs(H)$, and these test statistics are not independent in general because they all depend on the random variables $X_i$, $I_i$, $i=1,\dots,n$; furthermore, $T_n(\pi)$ cannot be chosen as a standard test statistic  because it must  be  suitable for the incomplete  data $(\pi(I_i),\pi(I_i\odot X_i))$, $i=1,\dots,n$.  For that reason, the development and analysis of our test statistic $T_n$ requires independent investigations. 

\medskip
Orienting on existing bootstrap procedures, such as they are discussed in \citet{Gaigall03102021}, \citet{BARINGHAUS2023105154} and  \citet{Gaigall-ApplicabilityServeralTests} for the  testing problems of goodness-of-fit, symmetry, exchangeability, and independence, it is of interest to design new resampling procedures  that are useful and efficient also in the case of incomplete observations, where common techniques fail to work. Following, properties of the newly developed tests have to be investigated.  Asymptotic aspects as the sample size (and possibly simultaneously the dimension) of the observations tends to infinity have to be studied, including distributional or almost sure limits of the novel test statistics under the null hypothesis or under alternatives, the asymptotic size of the tests or the asymptotic power, in particular the consistency, of the tests. For the mathematical investigations, it is possible to  rely on techniques developed in our previous works: In \citet{DitzhausGaigallConsistentTest}, \citet{DisthausGaigallMarginalHomo}, and \citet{GaigallLiangWu-HilbertIndependence}, methods from the theory of $U$-statistics are used, see \citet{Koroljuk}. In \citet{GaigallWuebbolding_gBM-Test}, the application of limit results available for general Hilbert spaces takes place, see \citet{LahaRohatgi}.  

\medskip
Monte-Carlo simulation studies can be used for the investigation of size and power of the new tests in the finite sample case. Thereby, novel insights enables the reduction of the number of replications in the bootstrap procedures in use and finally to speed up the simulations, see \citet{Gaigall08012026}. Because we deal with high or even infinite dimensional situations, the simulations are  computationally challenging and the development of  efficient testing procedures in this regard is requested. Closed-form formulas for the test statistics in use are needed.


\begin{thebibliography}{38}
\providecommand{\natexlab}[1]{#1}
\providecommand{\url}[1]{\texttt{#1}}
\expandafter\ifx\csname urlstyle\endcsname\relax
  \providecommand{\doi}[1]{doi: #1}\else
  \providecommand{\doi}{doi: \begingroup \urlstyle{rm}\Url}\fi

\bibitem[Aleksić and Milošević(2024)]{Aleksić09122024}
Danijel~G. Aleksić and Bojana Milošević.
\newblock To impute or not? testing multivariate normality on incomplete dataset: revisiting the bhep test.
\newblock \emph{Journal of Applied Statistics}, pages 1--18, 2024.
\newblock \doi{10.1080/02664763.2024.2438798}.

\bibitem[Bakirov et~al.(2006)Bakirov, Rizzo, and Székely]{BakirovRizzoSzekely-Inditest-06}
Nail~K. Bakirov, Maria~L. Rizzo, and Gábor~J. Székely.
\newblock A multivariate nonparametric test of independence.
\newblock \emph{Journal of Multivariate Analysis}, 97\penalty0 (8):\penalty0 1742--1756, 2006.
\newblock \doi{10.1016/j.jmva.2005.10.005}.

\bibitem[Baringhaus and Gaigall(2023)]{BARINGHAUS2023105154}
Ludwig Baringhaus and Daniel Gaigall.
\newblock A goodness-of-fit test for the compound poisson exponential model.
\newblock \emph{Journal of Multivariate Analysis}, 195:\penalty0 105154, 2023.
\newblock \doi{10.1016/j.jmva.2022.105154}.

\bibitem[Baringhaus and Henze(1988)]{BaringhausHenze}
Ludwig Baringhaus and Norbert Henze.
\newblock {A consistent test for multivariate normality based on the empirical characteristic function}.
\newblock \emph{Metrika: International Journal for Theoretical and Applied Statistics}, 35:\penalty0 339--348, 1988.
\newblock \doi{10.1007/BF02613322}.

\bibitem[Bugni and Horowitz(2021)]{https://doi.org/10.1002/jae.2846}
Federico~A. Bugni and Joel~L. Horowitz.
\newblock Permutation tests for equality of distributions of functional data.
\newblock \emph{Journal of Applied Econometrics}, 36:\penalty0 861--877, 2021.
\newblock \doi{10.1002/jae.2846}.

\bibitem[Bugni et~al.(2009)Bugni, Hall, Horowitz, and Neumann]{9e51bff2-c3de-37cf-b969-bf2517f82c90}
Federico~A. Bugni, Peter Hall, Joel~L. Horowitz, and George~R. Neumann.
\newblock Goodness-of-fit tests for functional data.
\newblock \emph{The Econometrics Journal}, 12:\penalty0 S1--S18, 2009.
\newblock \doi{10.1111/j.1368-423X.2008.00266.x}.

\bibitem[Cuesta-Albertos and Febrero-Bande(2010)]{RePEc:spr:testjl:v:19:y:2010:i:3:p:537-557}
Juan~Antonio Cuesta-Albertos and Manuel Febrero-Bande.
\newblock {A simple multiway ANOVA for functional data}.
\newblock \emph{TEST: An Official Journal of the Spanish Society of Statistics and Operations Research}, 19:\penalty0 537--557, 2010.
\newblock \doi{10.1007/s11749-010-0185-3}.

\bibitem[Cuesta-Albertos et~al.(2006)Cuesta-Albertos, Fraiman, and Ransford]{Random_projections}
Juan~Antonio Cuesta-Albertos, Ricardo Fraiman, and Thomas Ransford.
\newblock Random projections and goodness-of-fit tests in infinite-dimensional spaces.
\newblock \emph{Bull Braz Math Soc}, 37:\penalty0 477--501, 2006.
\newblock \doi{10.1007/s00574-006-0023-0}.

\bibitem[Cuesta-Albertos et~al.(2007)Cuesta-Albertos, {del Barrio}, Fraiman, and Matrán]{CUESTAALBERTOS20074814}
Juan~Antonio Cuesta-Albertos, Eustasio {del Barrio}, Ricardo Fraiman, and Carlos Matrán.
\newblock The random projection method in goodness of fit for functional data.
\newblock \emph{Computational Statistics \& Data Analysis}, 51:\penalty0 4814--4831, 2007.
\newblock \doi{10.1016/j.csda.2006.09.007}.

\bibitem[Cuevas and Fraiman(2009)]{RePEc:eee:jmvana:v:100:y:2009:i:4:p:753-766}
Antonio Cuevas and Ricardo Fraiman.
\newblock {On depth measures and dual statistics. A methodology for dealing with general data}.
\newblock \emph{Journal of Multivariate Analysis}, 100:\penalty0 753--766, 2009.
\newblock \doi{10.1016/j.jmva.2008.08.002}.

\bibitem[Cuparić and Milošević(2024)]{CuparicMilosevic-adapt-impute}
Marija Cuparić and Bojana Milošević.
\newblock {To impute or to adapt? Model specification tests’ perspective}.
\newblock \emph{Statistical Papers}, 65:\penalty0 1021--1039, 2024.
\newblock \doi{10.1007/s00362-023-01421-4}.

\bibitem[Delaigle and Meister(2021)]{0fdf3827-ba20-3d40-8bdd-32b0388081ee}
Aurore Delaigle and Alexander Meister.
\newblock Nonparametric density estimation for intentionally corrupted functional data.
\newblock \emph{Statistica Sinica}, 31\penalty0 (4):\penalty0 pp. 1915--1934, 2021.
\newblock \doi{10.5705/ss.202018.0484}.

\bibitem[Detke et~al.(2004)Detke, Wiltse, Mallinckrodt, McNamara, Demitrack, and Bitter]{DETKE2004457}
Michael~J. Detke, Curtis~G. Wiltse, Craig~H. Mallinckrodt, Robert~K. McNamara, Mark~A. Demitrack, and Istvan Bitter.
\newblock Duloxetine in the acute and long-term treatment of major depressive disorder: a placebo- and paroxetine-controlled trial.
\newblock \emph{European Neuropsychopharmacology}, 14:\penalty0 457--470, 2004.
\newblock \doi{10.1016/j.euroneuro.2004.01.002}.

\bibitem[Dette and Tang(2026)]{dette2024newenergydistancesstatistical}
Holger Dette and Jiajun Tang.
\newblock New energy distances for statistical inference on infinite dimensional hilbert spaces without moment conditions.
\newblock \emph{Bernoulli}, 2026.
\newblock URL \url{https://www.e-publications.org/ims/submission/BEJ/user/submissionFile/64170?confirm=122adcd9}.

\bibitem[Ditzhaus and Gaigall(2018)]{DitzhausGaigallConsistentTest}
Marc Ditzhaus and Daniel Gaigall.
\newblock A consistent goodness-of-fit test for huge dimensional and functional data.
\newblock \emph{Journal of Nonparametric Statistics}, 30\penalty0 (4):\penalty0 834--859, 2018.
\newblock \doi{10.1080/10485252.2018.1486402}.

\bibitem[Ditzhaus and Gaigall(2022)]{DisthausGaigallMarginalHomo}
Marc Ditzhaus and Daniel Gaigall.
\newblock Testing marginal homogeneity in hilbert spaces with applications to stock market returns.
\newblock \emph{TEST}, 31:\penalty0 749--770, 2022.
\newblock \doi{10.1007/s11749-022-00802-5}.

\bibitem[Ebner and Henze(2020)]{RePEc:spr:testjl:v:29:y:2020:i:4:d:10.1007_s11749-020-00740-0}
Bruno Ebner and Norbert Henze.
\newblock {Tests for multivariate normality—a critical review with emphasis on weighted $L^2$-statistics}.
\newblock \emph{TEST: An Official Journal of the Spanish Society of Statistics and Operations Research}, 29:\penalty0 845--892, 2020.
\newblock \doi{10.1007/s11749-020-00740-0}.

\bibitem[Epps and Pulley(1983)]{Epps-Pulley}
T.~W. Epps and Lawrence~B. Pulley.
\newblock A test for normality based on the empirical characteristic function.
\newblock \emph{Biometrika}, 70\penalty0 (3):\penalty0 723--726, 1983.
\newblock ISSN 00063444.
\newblock \doi{10.2307/2336512}.

\bibitem[Escobar et~al.(2014)Escobar, Klotz, Valdes, and Agudelo]{Escobar}
Juan~S. Escobar, Bernadette Klotz, Beatriz~E. Valdes, and Gloria~M. Agudelo.
\newblock The gut microbiota of colombians differs from that of americans, europeans and asians.
\newblock \emph{BMC Microbiology}, 14:\penalty0 311, 2014.
\newblock \doi{10.1186/s12866-014-0311-6}.

\bibitem[Fong et~al.(2017)Fong, Huang, Lemos, and Mcelrath]{10.1093/biostatistics/kxx039}
Youyi Fong, Ying Huang, Maria~P Lemos, and M~Juliana Mcelrath.
\newblock Rank-based two-sample tests for paired data with missing values.
\newblock \emph{Biostatistics}, 19:\penalty0 281--294, 2017.
\newblock \doi{10.1093/biostatistics/kxx039}.

\bibitem[Gaigall(2020)]{Gaigall-MargHom-Incompl-data}
Daniel Gaigall.
\newblock {Testing marginal homogeneity of a continuous bivariate distribution with possibly incomplete paired data}.
\newblock \emph{Metrika: International Journal for Theoretical and Applied Statistics}, 83:\penalty0 437--465, 2020.
\newblock \doi{10.1007/s00184-019-00742-5}.

\bibitem[Gaigall(2021)]{Gaigall03102021}
Daniel Gaigall.
\newblock On a new approach to the multi-sample goodness-of-fit problem.
\newblock \emph{Communications in Statistics - Simulation and Computation}, 50:\penalty0 2971--2989, 2021.
\newblock \doi{10.1080/03610918.2019.1618472}.

\bibitem[Gaigall(2023)]{Gaigall-ApplicabilityServeralTests}
Daniel Gaigall.
\newblock On the applicability of several tests to models with not identically distributed random effects.
\newblock \emph{Statistics}, 57:\penalty0 300--327, 2023.
\newblock \doi{10.1080/02331888.2023.2193748}.

\bibitem[Gaigall and Gerstenberg(2026)]{Gaigall08012026}
Daniel Gaigall and Julian Gerstenberg.
\newblock On the number of replications in resampling tests and monte carlo simulation studies.
\newblock \emph{The American Statistician}, pages 1--20, 2026.
\newblock \doi{10.1080/00031305.2025.2612197}.

\bibitem[Gaigall and Wübbolding(2025)]{GaigallWuebbolding_gBM-Test}
Daniel Gaigall and Philipp Wübbolding.
\newblock A goodness-of-fit test for geometric {Brownian} motion.
\newblock \emph{Computational Statistics \& Data Analysis}, 210:\penalty0 108196, 2025.
\newblock \doi{10.1016/j.csda.2025.108196}.

\bibitem[Gaigall and Wübbolding(2026)]{g+w-neu}
Daniel Gaigall and Philipp Wübbolding.
\newblock A {BHEP} test for multivariate normality on incomplete data, 2026.
\newblock URL \url{https://arxiv.org/abs/2607.03335}.

\bibitem[Gaigall et~al.(2025)Gaigall, Wu, and Liang]{GaigallLiangWu-HilbertIndependence}
Daniel Gaigall, Shunyao Wu, and Hua Liang.
\newblock A general approach for testing independence in hilbert spaces.
\newblock \emph{Journal of Multivariate Analysis}, 206:\penalty0 105384, 2025.
\newblock \doi{10.1016/j.jmva.2024.105384}.

\bibitem[Goldstein et~al.(2004)Goldstein, Lu, Detke, Wiltse, Mallinckrodt, and Demitrack]{Goldstein}
David~J. Goldstein, Yili Lu, Michael~J. Detke, Curtis Wiltse, Craig Mallinckrodt, and Mark~A. Demitrack.
\newblock Duloxetine in the treatment of depression: a double-blind placebo-controlled comparison with paroxetine.
\newblock \emph{Journal of Clinical Psychopharmacology}, 24:\penalty0 389--399, 2004.
\newblock \doi{10.1097/01.jcp.0000132448.65972.d9}.

\bibitem[Gonz{\'a}lez-Manteiga(2022)]{GonzlezManteiga2022ARO}
Wenceaslao Gonz{\'a}lez-Manteiga.
\newblock A review on specification tests for models with functional data.
\newblock \emph{Spanish Journal of Statistics}, pages 9--40, 2022.
\newblock \doi{10.37830/SJS.2022.1.02}.

\bibitem[Henze and Jiménez-Gamero(2020)]{HenzeGamero}
Norbert Henze and María~Dolores Jiménez-Gamero.
\newblock A test for gaussianity in hilbert spaces via the empirical characteristic functional.
\newblock \emph{Scandinavian Journal of Statistics}, 48:\penalty0 406--428, 05 2020.
\newblock \doi{10.1111/sjos.12470}.

\bibitem[Henze et~al.(2003)Henze, Klar, and Meintanis]{HENZE2003275}
Norbert Henze, Bernhard Klar, and Simos~G. Meintanis.
\newblock Invariant tests for symmetry about an unspecified point based on the empirical characteristic function.
\newblock \emph{Journal of Multivariate Analysis}, 87:\penalty0 275--297, 2003.
\newblock \doi{10.1016/S0047-259X(03)00044-7}.

\bibitem[Koroljukand and Borovskich(1994)]{Koroljuk}
V.~S. Koroljukand and X.V. Borovskich.
\newblock Theory of $u$-statistics.
\newblock Dordrecht: Kluwer Academic Publishers Group, 1994.
\newblock \doi{10.1007/978-94-017-3515-5}.

\bibitem[Kraus(2015)]{repec:bla:jorssb:v:77:y:2015:i:4:p:777-801}
David Kraus.
\newblock Components and completion of partially observed functional data.
\newblock \emph{Journal of the Royal Statistical Society Series B}, 77:\penalty0 777--801, 2015.
\newblock \doi{10.1111/rssb.12087}.

\bibitem[Laha and Rohatgi(1979)]{LahaRohatgi}
R.~G. Laha and V.~K. Rohatgi.
\newblock Probability theory.
\newblock In \emph{Wiley Series in Probability and Mathematical Statistics}. John Wiley \& Sons, Ltd, 1979.

\bibitem[Meintanis et~al.(2016)Meintanis, Allison, and Santana]{RePEc:spr:stpapr:v:57:y:2016:i:4:d:10.1007_s00362-016-0760-0}
Simos~G. Meintanis, James Allison, and Leonard Santana.
\newblock {Goodness-of-fit tests for semiparametric and parametric hypotheses based on the probability weighted empirical characteristic function}.
\newblock \emph{Statistical Papers}, 57:\penalty0 957--976, 2016.
\newblock \doi{10.1007/s00362-016-0760-0}.

\bibitem[Ragot and Ruiz(2008)]{10.1063/1.2981526}
Sébastien Ragot and María~Belén Ruiz.
\newblock Fourier–legendre expansion of the one-electron density matrix of ground-state two-electron atoms.
\newblock \emph{The Journal of Chemical Physics}, 129:\penalty0 124117, 2008.
\newblock \doi{10.1063/1.2981526}.

\bibitem[Romano(1989)]{Romano-Bootstrap-Rand-nonpar-89}
Joseph~P. Romano.
\newblock {Bootstrap and Randomization Tests of some Nonparametric Hypotheses}.
\newblock \emph{The Annals of Statistics}, 17:\penalty0 141 -- 159, 1989.
\newblock \doi{10.1214/aos/1176347007}.

\bibitem[Čoupek et~al.(2024)Čoupek, Dolník, Hlávka, and Hlubinka]{CDHH}
Petr Čoupek, Viktor Dolník, Zdeněk Hlávka, and Daniel Hlubinka.
\newblock Fourier approach to goodness-of-fit tests for {Gaussian} random processes.
\newblock \emph{Statistical Papers}, 65:\penalty0 2937--2972, 2024.
\newblock \doi{10.1007/s00362-023-01510-4}.

\end{thebibliography}
\end{document}